\newif\ifdm
\newif\ifarxiv
\dmfalse\arxivtrue

\ifarxiv
\pdfoutput=1
\documentclass{article}
\usepackage{amsmath, amsthm, amssymb}
\usepackage{enumitem}
\usepackage{geometry}
\usepackage[unicode,bookmarksnumbered=true,colorlinks=true,allcolors=blue,linktoc=all]{hyperref}
\else
\documentclass{article}
\usepackage[journal=DM,lang=british]{ems-journal}
\fi

\usepackage{eucal}
\usepackage[capitalise,nameinlink]{cleveref}
\usepackage{tikz-cd}
\usepackage{quiver}
\usepackage{stmaryrd}
\usepackage{listings}
\usepackage{etoolbox}
\usepackage{thm-restate}
\usepackage[nottoc,notlot,notlof]{tocbibind}
\usepackage{mathtools}
\usepackage{caption}

\crefname{subsection}{Subsection}{subsections}

\ifarxiv
\hypersetup{bookmarksdepth=2}
\usepackage{tocbasic}
\DeclareTOCStyleEntry[
  beforeskip=0pt,
  pagenumberformat=\textbf
]{tocline}{section}
\fi

\newtheorem{theorem}{Theorem}

\newtheorem{thm}{Theorem}
\newtheorem{lem}[thm]{Lemma}
\newtheorem{prop}[thm]{Proposition}

\theoremstyle{definition}
\newtheorem{defn}[thm]{Definition}

\theoremstyle{remark}
\newtheorem{remark}[thm]{Remark}

\newcommand{\ncmd}{\newcommand}

\definecolor{DefColor}{rgb}{0.6,0.15,0.25}
\ncmd{\mdef}[1]{\textcolor{DefColor}{#1}}
\ncmd{\tdef}[1]{\mdef{\emph{#1}}}

\DeclareRobustCommand{\minwidthbox}[2]{%
  \mathmakebox[\ifdim#2<\width\width\else#2\fi]{#1}%
}

\makeatletter
\newcommand{\xrightleftarrows}[2]{%
  \mathrel{\mathop{%
    \vcenter{\offinterlineskip\m@th
      \ialign{\hfil##\hfil\cr
        \hphantom{$\scriptstyle\mspace{8mu}{#1}\mspace{8mu}$}\cr
        \rightarrowfill\cr
        \vrule height0pt width 2em\cr
        \leftarrowfill\cr
        \hphantom{$\scriptstyle\mspace{8mu}{#2}\mspace{8mu}$}\cr
        \noalign{\kern-0.3ex}
      }%
    }%
  }\limits^{#1}_{#2}}%
}
\makeatother

\ncmd{\too}[1][]{\xrightarrow{\minwidthbox{#1}{1em}}}
\ncmd{\Too}[1][]{\xRightarrow{\minwidthbox{#1}{1em}}}
\ncmd{\adj}{\xrightleftarrows{\minwidthbox{}{1em}}{\minwidthbox{}{1em}}}
\ncmd{\iso}{\too[\smash{\raisebox{-0.5ex}{\ensuremath{\scriptstyle\sim}}}]}
\ncmd{\qin}{\quad\in\quad}

\ncmd{\mbb}[1]{\mathbb{#1}}
\ncmd{\mrm}[1]{\mathrm{#1}}
\ncmd{\mcl}[1]{\mathcal{#1}}
\ncmd{\mfk}[1]{\mathfrak{#1}}
\ncmd{\mbf}[1]{\mathbf{#1}}

\ncmd{\todo}[1]{\textbf{TODO #1}}
\ncmd{\reftodo}[1]{\textbf{REF #1}}

\ncmd{\id}[1][]{\mrm{id}\ifstrempty{#1}{}{_{#1}}}

\ncmd{\BB}{\mrm{B}}
\ncmd{\unit}{\mathbf{1}}
\ncmd{\CC}{\mcl{C}}
\ncmd{\DD}{\mcl{D}}
\ncmd{\EE}{\mcl{E}}
\ncmd{\bbE}{\mbb{E}}
\ncmd{\VV}{\mcl{V}}
\ncmd{\WW}{\mcl{W}}
\ncmd{\MM}{\mcl{M}}
\ncmd{\OO}{\mcl{O}}
\ncmd{\tA}{\mbf{A}}
\ncmd{\tB}{\mbf{B}}
\ncmd{\LM}{\mcl{LM}}
\ncmd{\Ass}{\mrm{Assoc}}
\ncmd{\II}{I}

\ncmd{\op}{\mrm{op}}
\ncmd{\co}{{\mrm{2}\text{-}\mrm{op}}}
\ncmd{\rev}{\mrm{rev}}
\ncmd{\atomicpres}{\mrm{at}}
\ncmd{\atomic}[1][]{\ifstrempty{#1}{}{#1\text{-}}\atomicpres}
\ncmd{\pt}{\mrm{pt}}
\ncmd{\Hei}{\mathrm{Heine}}
\ncmd{\lax}{\mathrm{lax}}
\ncmd{\laxV}{{\lax\text{-}\VV}}
\ncmd{\cl}{\mathrm{cl}}
\ncmd{\Idem}{\mrm{Idem}}
\ncmd{\sml}{\mrm{small}}

\ncmd{\Mod}{\mrm{Mod}}
\ncmd{\LMod}{\mrm{LMod}}
\ncmd{\Spaces}{\mcl{S}}
\ncmd{\SPACES}{\widehat{\Spaces}}
\ncmd{\Sp}{\mrm{Sp}}
\ncmd{\Set}{\mrm{Set}}
\ncmd{\Cat}{\mrm{Cat}}
\ncmd{\CatV}{\Cat^\VV}
\ncmd{\tCat}{\mbf{Cat}}
\ncmd{\CAT}{\widehat{\Cat}\vphantom{\Cat}}
\ncmd{\tCAT}{\widehat{\tCat}\vphantom{\tCat}}
\ncmd{\tCAAT}{\widehat{\widehat{\tCat}}\vphantom{\tCat}}
\ncmd{\CATsml}{\CAT_\sml}
\ncmd{\tCATsml}{\tCAT_\sml}
\ncmd{\Catidem}{\Cat_{\mrm{idem}}}
\ncmd{\CatVcpl}{\CatV_{\mrm{cpl}}}
\ncmd{\Catperf}{\Cat_{\mrm{perf}}}
\renewcommand{\Pr}{\mrm{Pr}}
\ncmd{\tPr}{\mbf{Pr}}
\ncmd{\LL}{\mrm{L}}
\ncmd{\RR}{\mrm{R}}
\ncmd{\PrL}{\Pr^\LL}
\ncmd{\PrLL}{\Pr^{\mrm{LL}}}
\ncmd{\PrLLmol}{\Pr^{\mrm{LL},\mrm{mol}}}
\ncmd{\tPrL}{\tPr^\LL}
\ncmd{\tPrLL}{\tPr^{\mrm{LL}}}
\ncmd{\tPrLLmol}{\tPr^{\mrm{LL},\mrm{mol}}}
\ncmd{\PrLV}{\PrL_\VV}
\ncmd{\PrLLV}{\PrLL_\VV}
\ncmd{\PrLLmolV}{\PrLLmol_\VV}
\ncmd{\PrLLmolSp}{\PrLLmol_\Sp}
\ncmd{\PrLstw}{\PrL_{\mrm{st},\omega}}
\ncmd{\Mon}{\mrm{Mon}}

\ncmd{\yon}{\text{\usefont{U}{min}{m}{n}\symbol{'110}}}
\DeclareFontFamily{U}{min}{}
\DeclareFontShape{U}{min}{m}{n}{<-> dmjhira}{}
\ncmd{\cyon}{\yon{}^\mathnormal{c}}
\ncmd{\cwyon}{\widehat{\yon}{}^\mathnormal{c}}
\ncmd{\yonV}{\yon{}^\VV}

\ncmd{\wiota}{\widehat{\iota}}

\DeclareMathOperator{\eval}{eval}
\DeclareMathOperator{\Fun}{Fun}
\DeclareMathOperator{\tFun}{\mbf{Fun}}
\DeclareMathOperator{\tFunt}{\tFun_2}
\DeclareMathOperator{\Nat}{Nat}
\DeclareMathOperator{\Natt}{\Nat_2}

\DeclareMathOperator{\Alg}{Alg}
\DeclareMathOperator*{\colim}{colim}
\DeclareMathOperator*{\oplaxcolim}{oplaxcolim}

\DeclareMathOperator{\PSh}{\mcl{P}}
\DeclareMathOperator{\cPSh}{\mcl{P}^\mathnormal{c}}
\DeclareMathOperator{\cwPSh}{\widehat{\mcl{P}}{}^\mathnormal{c}}

\ncmd{\isml}{i_\sml}

\ncmd\noloc{%
  \nobreak
  \mspace{6mu plus 1mu}
  {:}
  \nonscript\mkern-\thinmuskip
  \mathpunct{}
  \mspace{2mu}
}

\title{Uniqueness and $(\infty,2)$-Naturality of Yoneda}
\ifarxiv
\author{Shay Ben-Moshe}
\date{}
\fi

\begin{document}
  \ifarxiv
	\maketitle
  \fi
	
	\begin{abstract}
    We show that the Yoneda embedding extends to an $(\infty,2)$-natural transformation.
    Furthermore, as such, it is uniquely determined by its value at the trivial $\infty$-category.
    We also study the naturality of the Yoneda lemma in its arguments, showing that it is an isomorphism of $(\infty,2)$-natural transformations.
	\end{abstract}

  \ifarxiv
	\tableofcontents

  \vspace{1em}
  \fi

  \begin{figure}[ht!]
    \centering
    \includegraphics[width=115mm]{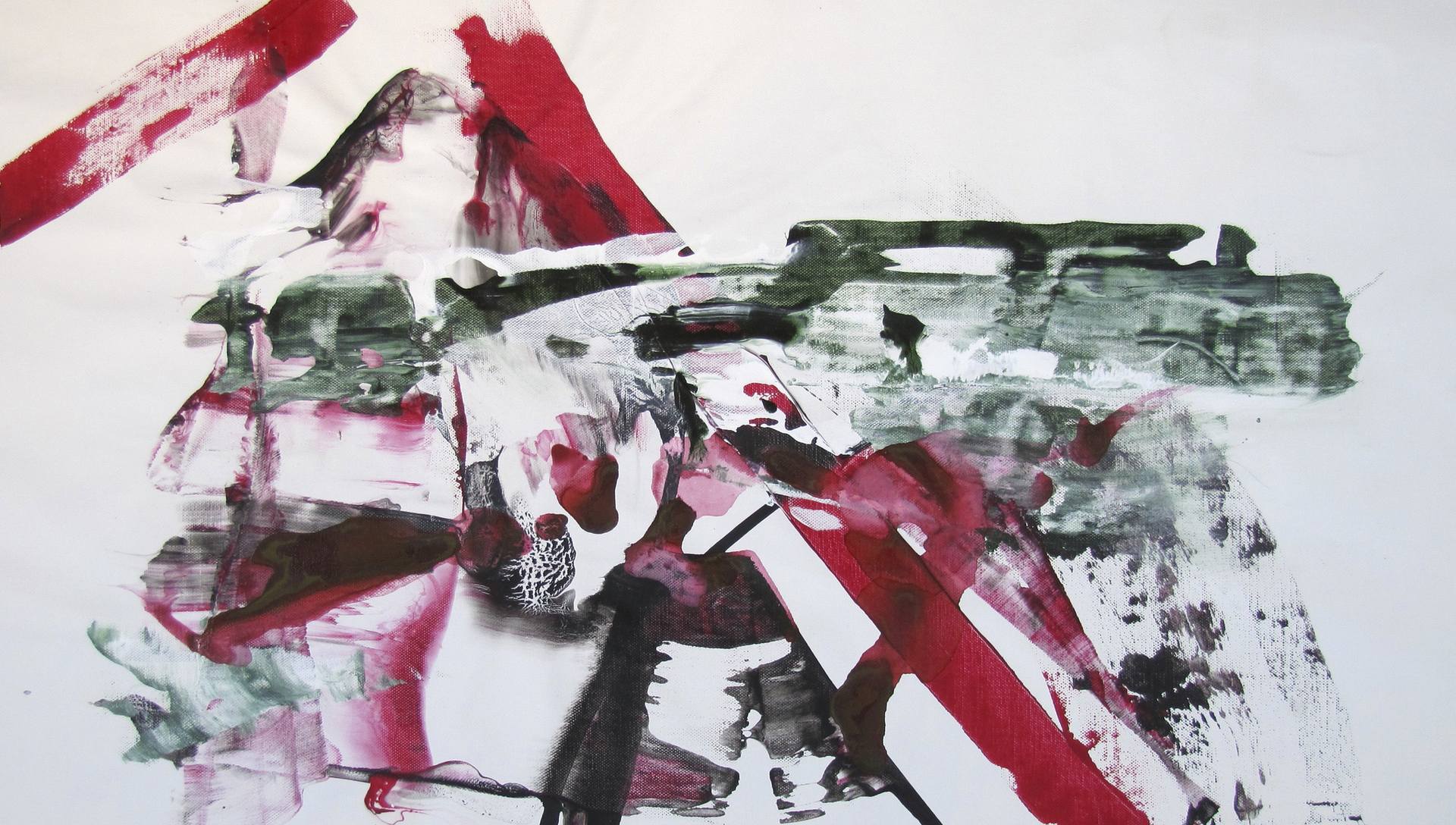}
    \caption*{
      ``\href{https://www.saatchiart.com/art/Painting-THE-RED-BLACK-CONCEPT-2-NATURALIZATION-ORIGINAL-MIXED-MEDIA-ABSTRACT-PAINTING/824166/3101086/view}{The Red/Black Concept 2- Naturalization}''
      by \href{https://www.saatchiart.com/Matkovsky}{Dmitri Matkovsky}.
    }
  \end{figure}

  \ifarxiv
  \pagebreak
  \fi

\section{Introduction}

Throughout the paper, we use the term category to mean an $\infty$-category and the term $2$-category to mean an $(\infty,2)$-category.

Given a small category $\CC$, we denote by $\PSh(\CC) := \Fun(\CC^\op, \Spaces)$ the large category of presheaves, and we denote the Yoneda embedding functor by
\[
    \yon\colon \CC \too \PSh(\CC),
    \qquad X \mapsto \hom(-, X).
\]
This construction is also characterized by a universal property: the Yoneda embedding exhibits $\PSh(\CC)$ as the free cocompletion of $\CC$. That is, for every large category $\EE$ admitting small colimits, pre-composition with Yoneda induces an isomorphism
\[
    \yon^*\colon \Fun_\sml(\PSh(\CC), \EE) \iso \Fun(\CC, \EE),
\]
where the source denotes the category of functors $\PSh(\CC) \to \EE$ preserving small colimits.
As we recall in \cref{sec-yon-emb}, this universal property makes the construction natural in $\CC$.
Namely, let $\iota\colon \Cat \to \CAT$ denote the inclusion of small categories into large categories, then the Yoneda embedding extends to a natural transformation
\[\begin{tikzcd}
	\Cat && {\CAT.}
	\arrow[""{name=0, anchor=center, inner sep=0}, "\iota", curve={height=-12pt}, from=1-1, to=1-3]
	\arrow[""{name=1, anchor=center, inner sep=0}, "\PSh"', curve={height=12pt}, from=1-1, to=1-3]
	\arrow["\yon", shorten <=3pt, shorten >=3pt, Rightarrow, from=0, to=1]
\end{tikzcd}\]


In this short paper we consider $2$-categorical aspects of Yoneda.
Recall that small categories arrange into a $2$-category $\tCat$, and similarly for large categories $\tCAT$, and $\iota$ enhances to a $2$-functor $\iota\colon \tCat \to \tCAT$.
We use the fact that the presheaves functor is symmetric monoidal, together with Heine's work \cite{Hei} on the relationship between enriched categories and tensored categories, to show that the presheaves functor and the Yoneda embedding admit a $2$-categorical refinement.
Furthermore, we use the $2$-categorical Yoneda lemma proven by Hinich \cite{Hin} to show that as a $2$-natural transformation, the Yoneda embedding is uniquely determined by its value at the trivial category.

\begin{theorem}[{\cref{yon-is-2} and \cref{nat-unique}}]\label{yon-emb-eu}
    The Yoneda embedding extends to a $2$-natural transformation
    \[\begin{tikzcd}
        \tCat && {\tCAT.}
        \arrow[""{name=0, anchor=center, inner sep=0}, "\iota", curve={height=-12pt}, from=1-1, to=1-3]
        \arrow[""{name=1, anchor=center, inner sep=0}, "\PSh"', curve={height=12pt}, from=1-1, to=1-3]
        \arrow["\yon", shorten <=3pt, shorten >=3pt, Rightarrow, from=0, to=1]
    \end{tikzcd}\]
    Furthermore, it is the unique $2$-natural transformation $\iota \Rightarrow \PSh$ whose value at the trivial category $\pt \in \tCat$ is the functor $\pt \to \PSh(\pt) \simeq \Spaces$ choosing the point.
\end{theorem}

One of the main features of the Yoneda embedding is the Yoneda lemma, typically phrased as follows.
Let $\CC$ be a small category and let $X \in \CC$, then, there is an isomorphism
\[
    \hom(\yon(X), F) \iso F(X)
\]
natural in presheaves $F\colon \CC^\op \to \Spaces$.

In \cref{coyoneda}, we show that the ($1$-)naturality of the Yoneda embedding implies the Yoneda lemma.
To connect the two, note that $\hom(\yon(X), -)$ is itself the Yoneda embedding of $\yon(X)$ as an object of $\PSh(\CC)^\op$.
To streamline the arguments, it will be convenient to work exclusively with copresheaves and the coYoneda embedding, namely replace $\CC$ by $\CC^\op$.
We denote the category of copresheaves by $\cPSh(\CC) := \Fun(\CC, \Spaces)$ and the coYoneda embedding by
\[
    \cyon\colon \CC \too \cPSh(\CC)^\op,
    \qquad X \mapsto \hom(X, -).
\]
Similarly, for a large category $\DD$, we can consider the huge category of large copresheaves $\cwPSh(\DD) := \Fun(\DD, \SPACES)$, together with its coYoneda embedding.
Applying this to $\DD = \cPSh(\CC)$, we obtain
\[
    \cwyon\colon \cPSh(\CC) \too \cwPSh(\cPSh(\CC))^\op.
\]
In these terms, the Yoneda lemma can be rephrased as an isomorphism
\[
    \cwyon(\cyon(X)) \simeq X^*,
\]
and we show that this follows from the naturality of the Yoneda embedding applied to the map $X\colon \pt \to \CC$.

As an application of the $2$-naturality of the Yoneda embedding from \cref{yon-emb-eu}, we make the Yoneda lemma natural in the other arguments.
The $2$-naturality allows us to vary the map $X\colon \pt \to \CC$, making the isomorphism natural in $X \in \CC$.
Further, it allows us to vary $\CC$, making the isomorphism natural in the category itself, namely, establishing an isomorphism between the double coYoneda embedding and the evaluation as $2$-natural transformations.
In fact, the same uniqueness argument from \cref{yon-emb-eu} applies in this case to show that there is a unique such $2$-natural transformation, which gives an alternative proof.
For the sake of completeness, we give both proofs in the body of the paper.

\begin{theorem}[{\cref{yon-lem-2} and \cref{nat-unique}}]\label{yon-lem-eu}
    The double coYoneda embedding and the evaluation are isomorphic as $2$-natural transformations filling the diagram
    \[\begin{tikzcd}
        \tCat && {\tCAAT.}
        \arrow[""{name=0, anchor=center, inner sep=0}, "{\wiota \circ \iota}", curve={height=-12pt}, from=1-1, to=1-3]
        \arrow[""{name=1, anchor=center, inner sep=0}, "\cwPSh\circ\cPSh"', curve={height=12pt}, from=1-1, to=1-3]
        \arrow[shorten <=3pt, shorten >=3pt, Rightarrow, from=0, to=1]
    \end{tikzcd}\]
    Furthermore, they are the unique $2$-natural transformation $\wiota \circ \iota \Rightarrow \cwPSh\circ\PSh$ whose value at the trivial category $\pt \in \tCat$ is the functor $\pt \to \cwPSh(\cPSh(\pt)) \simeq \Fun(\Spaces, \SPACES)$ choosing the inclusion of small spaces into large spaces.
\end{theorem}

\subsection*{Relation to Other Work}\label{psh-functorialities}

There is a different question regarding the naturality of the Yoneda embedding which was recently addressed in the literature.
The construction of the presheaves ($1$-)functor discussed in this paper is based on its universal property as the free cocompletion.
Yet, there is an alternative candidate for the presheaves functor.
Consider the functor $\Fun((-)^\op, \Spaces)$, sending a morphism $F\colon \CC \to \DD$ to
\[
    F^*\colon \Fun(\DD^\op, \Spaces) \to \Fun(\CC^\op, \Spaces)
\]
given by pre-composition with $F^\op$.
The functor $F^*$ admits a left adjoint $F_!$, given by left Kan extension along $F^\op$.
Passing to the left adjoint functoriality of $\Fun((-)^\op, \Spaces)$ thus gives rise to a functor from small categories to large categories.
One can show that for any $F$ as above, $F_!$ agrees with the value of $\PSh$ at $F$, namely, that the two constructions of the presheaves functor agree on morphisms.
However, since we are working in higher category theory, it does not immediately follow that the two constructions agree coherently, or that the Yoneda embedding is even natural with respect to this other construction.
These questions were recently raised and answered in the affirmative by Haugseng--Hebestreit--Linskens--Nuiten \cite{HHLN}, and were subsequently resolved by Ramzi \cite{Ram} by different means.

The present paper leaves open analogous questions in the $2$-categorical setting.
Our enhancement of $\PSh$ to a $2$-functor relies on its symmetric monoidal structure, stemming from its role as the free cocompletion.
Alternatively, one can try to enhance $\Fun((-)^\op, \Spaces)$ into a $2$-functor, and as above, pass to the left adjoint functoriality.
It is not clear to us that the two enhancements would agree, or that the Yoneda embedding would be $2$-natural for such an enhancement.

In a previous paper \cite{EnrNat}, we have addressed the question of the naturality of the enriched Yoneda embedding.
Similarly to the present paper, this was established using the universal property of enriched presheaves as the free cocompletion.
The question of the comparison to the other construction of enriched presheaves, originally raised in \cite[Question 1.5]{EnrNat}, was recently answered by Heine \cite[Theorem 1.5 and Corollary 1.6]{BiEnr}.
The question of the $2$-naturality in the enriched setting from \cite[Question 1.7]{EnrNat} remains open.

\ifarxiv
\subsection*{Acknowledgements}

I would like to thank Emmanuel Farjoun for his encouragement to write this paper, and for helpful comments on an early draft.
I also thank Lior Yanovski and Shai Keidar for helpful conversations.
I thank the anonymous referee for carefully reading the manuscript and for their helpful comments.
I also thank Vesna Stojanoska for her suggestions throughout the review process.
Finally, I thank Dmitri Matkovsky for allowing me to use his painting as the cover image.

\fi

\section{The Yoneda Embedding}\label{sec-yon-emb}

In this section we show that the Yoneda embedding extends to a $2$-natural transformation.
We begin by collecting some facts about presheaves and Yoneda.
Although the following way to handle size issues when trying to view presheaves and the Yoneda embedding as part of an adjunction appears in \cite{HTT} and is known to experts, it is not widely familiar.
Therefore, we include a brief discussion for the reader's convenience.

In \cite[\S5.3.6]{HTT}, Lurie studies the operation of adjoining colimits to categories.
Given a \emph{large} category $\CC$, there is another large category $\PSh(\CC)$, together with a functor $\yon\colon \CC \to \PSh(\CC)$, exhibiting $\PSh(\CC)$ as given by freely adjoining all small colimits to $\CC$.
We warn the reader that for a general large category, $\PSh(\CC)$ is \emph{not} a category of presheaves on $\CC$.
Rather, it is the full subcategory of the category of large presheaves $\Fun(\CC^\op, \SPACES)$, given by the closure of representables under \emph{small} colimits.
However, when $\CC$ is small, we see that it is the category of small presheaves $\Fun(\CC^\op, \Spaces)$.

The universal property of this construction can be phrased as follows.
Consider $\CATsml$, the category of large categories admitting small colimits and functors preserving them.
Then, the inclusion $\isml\colon \CATsml \to \CAT$ admits a left adjoint given point-wise by $\PSh(\CC)$ with unit given point-wise by $\yon$.

In \cite[\S4.8.1]{HA}, Lurie studies the tensor product of categories, and in particular shows that the adjunction is symmetric monoidal, namely that $\PSh$ is symmetric monoidal, $\isml$ is lax symmetric monoidal, and the Yoneda embedding is a map of lax symmetric monoidal functors.

We summarize these results in the following proposition.

\begin{prop}\label{PShsmall-sm-adj}
    There is a symmetric monoidal adjunction
    \[
        \PSh\colon \CAT \adj \CATsml\colon \isml
    \]
    such that for a small category $\CC$ we have $\PSh(\CC) \simeq \Fun(\CC^\op, \Spaces)$ and the unit map is given by the Yoneda embedding.
\end{prop}

\begin{proof}
    The fact that $\PSh$ is left adjoint to $\isml$ is \cite[Corollary 5.3.6.10]{HTT}.
    The description of the value when $\CC$ is small is \cite[Example 5.3.6.4]{HTT}.
    The fact that the left adjoint $\PSh$ is symmetric monoidal is a consequence of \cite[Proposition 4.8.1.3]{HA} (see also \cite[Remark 4.8.1.8]{HA}).
\end{proof}

We now use this to pass to the $2$-categorical setting.
Recall that there are various models for $2$-categories, which we choose to model by categories enriched in categories in the sense of Gepner--Haugseng \cite{GH} or, equivalently, Hinich \cite{Hin} (see \cite[\S2]{haugseng2021lax} for a recent discussion of models of $2$-categories).

One source of enriched categories is tensored categories, as was initially shown in \cite[Corollary 7.4.9]{GH}, and was vastly expanded by Heine \cite{Hei}.
We give a brief overview for the convenience of the reader.
Fix a presentably monoidal category $\VV \in \Alg(\PrL)$.
Let $\CC$ be a category tensored over $\VV$.
Assume that the action is closed, namely, that for every $X \in \CC$ the functor $- \otimes X\colon \VV \to \CC$ admits a right adjoint, denoted $\hom^\VV(X, -)\colon \CC \to \VV$.
Then, these right adjoints assemble into the structure of a $\VV$-enriched category on $\CC$.
Furthermore, if $\CC$ and $\DD$ are closed $\VV$-tensored categories, then a lax $\VV$-linear functor $F\colon \CC \to \DD$ induces a $\VV$-enriched functor between the corresponding $\VV$-enriched categories.
One can then observe that if $F$ is (strong) $\VV$-linear and a left adjoint, its right adjoint is lax $\VV$-linear, and thus the adjunction induces an adjunction of $\VV$-enriched categories.
Using (the large version of) this observation applied to $\VV = \CAT$, we can make the passage to $2$-categories.

\begin{prop}\label{PShsmall-2-adj}
    The adjunction from \cref{PShsmall-sm-adj} enhances to an adjunction in the $2$-category of $2$-categories
    \[
        \PSh\colon \tCAT \adj \tCATsml\colon \isml.
    \]
\end{prop}

\begin{proof}
    The symmetric monoidal structure on $\PSh$ from \cref{PShsmall-sm-adj} makes $\CATsml$ into a $\CAT$-tensored category (since a commutative algebra is in particular a module), for which $\PSh$ is $\CAT$-linear.
    Furthermore, since the symmetric monoidal structure on $\CATsml$ is closed, and $\PSh$ admits a right adjoint, the action of $\CAT$ on $\CATsml$ is closed.
    As explained above, \cite[Theorem 1.2]{Hei} (see also \cite[Theorem 1.8]{Hei-arxiv}\footnote{Note that the $2$-categorical enhancement of $\chi$ does not appear in the published version of Heine's paper \cite{Hei}, but does appear in a newer arXiv version \cite{Hei-arxiv}.} and \cite[Corollary 4.14]{EnrNat} for a more elaborate discussion) implies that we get an induced adjunction of $\CAT$-enriched categories, i.e.\ of $2$-categories.
\end{proof}

This readily implies the existence of a $2$-categorical lift of the Yoneda embedding.

\begin{prop}\label{yon-is-2}
    The Yoneda embedding extends to a $2$-natural transformation
    \[\begin{tikzcd}
        \tCat && {\tCAT.}
        \arrow[""{name=0, anchor=center, inner sep=0}, "\iota", curve={height=-12pt}, from=1-1, to=1-3]
        \arrow[""{name=1, anchor=center, inner sep=0}, "\PSh"', curve={height=12pt}, from=1-1, to=1-3]
        \arrow["\yon", shorten <=3pt, shorten >=3pt, Rightarrow, from=0, to=1]
    \end{tikzcd}\]
\end{prop}

\begin{proof}
    The unit map of the adjunction in \cref{PShsmall-2-adj} provides us with a $2$-natural transformation which we pre-compose with $\iota$ as follows
    \[\begin{tikzcd}
        \tCat & \tCAT && {\tCAT.} \\
        && \tCATsml
        \arrow[""{name=0, anchor=center, inner sep=0}, "{\id[\tCAT]}", from=1-2, to=1-4]
        \arrow["\iota", from=1-1, to=1-2]
        \arrow["\PSh"', curve={height=12pt}, from=1-2, to=2-3]
        \arrow["\isml"', curve={height=12pt}, from=2-3, to=1-4]
        \arrow["\yon", shorten <=3pt, shorten >=3pt, Rightarrow, from=0, to=2-3]
    \end{tikzcd}\]
\end{proof}

\section{The Yoneda Lemma}

In this section we show that the naturality of the Yoneda embedding implies the Yoneda lemma.
As was mentioned in the introduction, in the context of the Yoneda lemma, it is more convenient to work consistently with copresheaves.
We note that the copresheaves functor has a contravariant $2$-functoriality.
Indeed, recall that the construction $\CC \mapsto \CC^\op$ is a $2$-equivalence
\[
    (-)^\op\colon \tCat \iso \tCat^\co,
\]
and we give the following definition.

\begin{defn}
    We define the copresheaves $2$-functor to be
    \[
        \cPSh\colon \tCat \too[(-)^\op] \tCat^\co \too[\PSh] \tCAT^\co,
    \]
    equipped with the coYoneda embedding $2$-natural transformation
    \[
        \cyon\colon \iota \Too \cPSh(-)^\op.
    \]
\end{defn}

Recall that by \cite[8.1]{HHLN} or \cite{Ram}, the underlying ($1$-)functor of $\cPSh$ is given by passing to the left adjoint functoriality of $\Fun(-, \Spaces)$.
Namely, it sends a small category $\CC$ to the large category $\Fun(\CC, \Spaces)$, and $F\colon \CC \to \DD$ to the left Kan extension $F_!\colon \Fun(\CC, \Spaces) \to \Fun(\DD, \Spaces)$, whose right adjoint is the pre-composition $F^*$.

Consider the functors $F_!$ and $F^*$.
Applying the large version of copresheaves, we obtain two functors $F_{!!}$ and $F^{**}$, both from $\cwPSh(\cPSh(\CC))$ to $\cwPSh(\cPSh(\DD))$.
We begin by showing that they coincide.

\begin{lem}\label{!!-**}
    Let $F\colon \CC \to \DD$ be a functor, then there is a natural isomorphism $F_{!!} \simeq F^{**}$ of functors $\cwPSh(\cPSh(\CC)) \to \cwPSh(\cPSh(\DD))$.
\end{lem}

\begin{proof}
    Consider the adjunction $F_! \dashv F^*$.
    Applying $\cwPSh$, and recalling that it has contravariant $2$-functoriality and thus reverses the adjunction, we get an adjunction $(F^*)_! \dashv F_{!!}$.
    Now, consider the functor $F^*$ itself, to which there is an associated adjunction $(F^*)_! \dashv F^{**}$.
    The uniqueness of adjoints thus implies that $F_{!!} \simeq F^{**}$.
\end{proof}

\begin{remark}
    This can also be seen in more direct (and less coherent) terms using the formula for left Kan extension.
    Given $\varphi \in \cwPSh(\cPSh(\CC))$ and $g \in \cPSh(\DD)$, we have
    \[
        F_{!!}(\varphi)(g)
        \simeq \colim_{F_!(f) \to g} \varphi(f)
        \simeq \colim_{f \to F^*(g)} \varphi(f)
        \simeq \varphi(F^*(g))
        \simeq F^{**}(\varphi)(g).
    \]
\end{remark}

With this in mind, we can prove the Yoneda lemma.

\begin{prop}\label{coyoneda}
    For a small category $\CC$ and $X \in \CC$, there is a natural isomorphism
    \[
        \hom(\cyon(X), -) \simeq X^*
    \]
    of functors $\Fun(\CC, \Spaces) \to \Spaces$.
\end{prop}

\begin{proof}
    By the ($1$-)naturality of the Yoneda embedding, we have the following commutative diagram
    \[\begin{tikzcd}
        \pt & {\cPSh(\pt)^\op} & {\cwPSh(\cPSh(\pt))} \\
        \CC & {\cPSh(\CC)^\op} & {\cwPSh(\cPSh(\CC)).}
        \arrow["\cyon", from=1-1, to=1-2]
        \arrow["X", from=1-1, to=2-1]
        \arrow["\cwyon", from=1-2, to=1-3]
        \arrow["{X_!}", from=1-2, to=2-2]
        \arrow["{X_{!!}}", from=1-3, to=2-3]
        \arrow["\cyon", from=2-1, to=2-2]
        \arrow["\cwyon", from=2-2, to=2-3]
    \end{tikzcd}\]
    Recall that $\cyon(\pt) \in \cPSh(\pt) \simeq \Spaces$ is the point, thus $\cwyon(\cyon(\pt)) \in \cwPSh(\cPSh(\pt)) \simeq \Fun(\Spaces, \SPACES)$ is the inclusion $\iota_\Spaces\colon \Spaces \hookrightarrow \SPACES$.
    By \cref{!!-**}, we have $X_{!!} \simeq X^{**}$.
    Thus, by the commutativity of the diagram we get
    \[
        \cwyon(\cyon(X))
        \simeq X_{!!}(\cwyon(\cyon(\pt)))
        \simeq X_{!!}(\iota_\Spaces)
        \simeq X^{**}(\iota_\Spaces)
        \simeq X^*.
    \]
\end{proof}

We move on to showing the $2$-naturality of the Yoneda lemma.
To that end, we first need to make the evaluation into a $2$-natural transformation.

\begin{defn}
    We define the evaluation $2$-natural transformation $\eval\colon \wiota \circ \iota \Rightarrow \cwPSh \circ \cPSh$ to be the composition
    \begin{align*}
        \wiota(\iota(-))
        &\iso \Fun(\pt, -)\\
        &\too \Fun(\cwPSh(\cPSh(\pt)), \cwPSh(\cPSh(-)))\\
        &\too \Fun(\pt, \cwPSh(\cPSh(-)))\\
        &\iso \cwPSh(\cPSh(-)),
    \end{align*}
    where the second map is applying the copresheaves functor twice and the third is by pre-composing with the double coYoneda embedding of the point $\cwyon \circ \cyon\colon \pt \to \cwPSh(\cPSh(\pt))$.
\end{defn}

We now extend the proof of \cref{coyoneda} to include the $2$-naturality.

\begin{prop}\label{yon-lem-2}
    There is an isomorphism
    \[
        \cwyon \circ \cyon \simeq \eval
    \]
    of $2$-natural transformations $\wiota \circ \iota \Too \cwPSh \circ \cPSh$.
\end{prop}

\begin{proof}
    The small and large versions of the $2$-naturality of the Yoneda embedding from \cref{yon-is-2} compose as follows
    \[\begin{tikzcd}
        \tCat && \tCAT && {\tCAAT.}
        \arrow[""{name=0, anchor=center, inner sep=0}, "\iota", curve={height=-12pt}, from=1-1, to=1-3]
        \arrow[""{name=1, anchor=center, inner sep=0}, "{\cPSh(-)^\op}"', curve={height=12pt}, from=1-1, to=1-3]
        \arrow[""{name=2, anchor=center, inner sep=0}, "\wiota", curve={height=-12pt}, from=1-3, to=1-5]
        \arrow[""{name=3, anchor=center, inner sep=0}, "{\cwPSh(-)^\op}"', curve={height=12pt}, from=1-3, to=1-5]
        \arrow["\cyon", shorten <=3pt, shorten >=3pt, Rightarrow, from=0, to=1]
        \arrow["\cwyon", shorten <=3pt, shorten >=3pt, Rightarrow, from=2, to=3]
    \end{tikzcd}\]
    The $2$-naturality of the composition $\cwyon \circ \cyon\colon \wiota \circ \iota \Rightarrow \cwPSh \circ \cPSh$ gives the commutativity of the square
    \[\begin{tikzcd}
        & {\Fun(\wiota(\iota(\pt)),\wiota(\iota(-)))} \\
        {\Fun(\pt,-)} && {\Fun(\wiota(\iota(\pt)),\cwPSh(\cPSh(-)))} \\
        & {\Fun(\cwPSh(\cPSh(\pt)),\cwPSh(\cPSh(-)))}
        \arrow[from=1-2, to=2-3]
        \arrow[from=2-1, to=1-2]
        \arrow[from=2-1, to=3-2]
        \arrow[from=3-2, to=2-3]
    \end{tikzcd}\]
    of $2$-functors from $\tCat$ to $\tCAAT$.
    Observe that the $2$-functor on the left is equivalent to $\wiota(\iota(-))$, and the $2$-functor on the right is equivalent to $\cwPSh(\cPSh(-))$.
    Moreover, by construction, the upper composition is $\cwyon \circ \cyon$ while the lower composition is $\eval$, concluding the proof.
\end{proof}

\section{Uniqueness}

We conclude by establishing the uniqueness of the $2$-natural Yoneda embedding, and the alternative proof for the identification of the double coYoneda embedding with the evaluation $2$-natural transformation.

\begin{defn}
    Given two $2$-functors $F, G\colon \tA \to \tB$, we denote by $\Natt(F, G)$ the category of $2$-natural transformations between them, defined to be the category of morphisms between them in the $2$-category of $2$-functors $\tFunt(\tA, \tB)$.
\end{defn}

Note that the following proposition is stated for $\tCAT$ and $\iota$, but the same holds with $\tCAAT$ and $\wiota \circ \iota$ in their place.

\begin{prop}\label{nat-unique}
    Let $F\colon \tCat \to \tCAT$ be a $2$-functor, then there is an equivalence of categories
    \[
        \Natt(\iota, F) \iso F(\pt)
        \qin \tCAT.
    \]
\end{prop}

\begin{proof}
    Recall from \cite[6.2.7]{Hin} (see also \cite[Theorem 2.10(1)]{EnrNat} for a reformulation closer to our context) that the $2$-Yoneda lemma for large categories says that for a large $2$-category $\tA$, an object $X \in \tA$ and a $2$-functor $F\colon \tA \to \tCAT$, evaluation at $\id[X]$ induces an equivalence
    \[
        \Natt(\hom(X, -), F) \iso F(X)
        \qin \tCAT.
    \]
    The claim then follows by taking $\tA = \tCat$ and noting that $\iota$ is corepresented by $\pt \in \tCat$.
\end{proof}

\begin{remark}
    This result can also be seen as a consequence of the fact that the $2$-category of categories is freely generated under oplax colimits from the trivial category $\pt$.
    We do not give a complete proof, but sketch the main idea.
    Let $\alpha\colon \iota \Rightarrow F$ be a $2$-natural transformation.
    Any $\CC \in \tCat$ can be written as the constant oplax colimit $\CC \simeq \oplaxcolim_\CC(\pt)$.
    Using the assembly map for oplax colimits, and the fact that $\iota$ commutes with oplax colimits, we get the following commutative diagram:
    \[\begin{tikzcd}
        {\oplaxcolim_\CC(\iota(\pt))} &&& {\oplaxcolim_\CC(F(\pt))} \\
        {\iota(\oplaxcolim_\CC(\pt))} &&& {F(\oplaxcolim_\CC(\pt))} \\
        {\iota(\CC)} &&& {F(\CC)}
        \arrow["{\alpha_{\oplaxcolim_\CC(\pt)}}", from=2-1, to=2-4]
        \arrow["\wr"', from=1-1, to=2-1]
        \arrow[from=1-4, to=2-4]
        \arrow["{\oplaxcolim_\CC(\alpha_{\pt})}", from=1-1, to=1-4]
        \arrow["\wr"', from=2-1, to=3-1]
        \arrow["{\alpha_{\CC}}", from=3-1, to=3-4]
        \arrow["\wr", from=2-4, to=3-4]
    \end{tikzcd}\]
    This expresses $\alpha_\CC$ as the composition of $\oplaxcolim_\CC(\alpha_\pt)$ and the assembly map of $F$, showing that $\alpha$ is determined by $\alpha_\pt$.
\end{remark}

  \ifarxiv
	\bibliographystyle{alpha}
  \else
  \bibliographystyle{emss}
  \fi
	\bibliography{refs}

\begin{thebibliography}{HHLN23}

\bibitem[BM24]{EnrNat}
Shay Ben-Moshe.
\newblock {Naturality of the $\infty$-categorical enriched Yoneda embedding}.
\newblock {\em J. Pure Appl. Algebra}, 228(6):107625, 2024.

\bibitem[GH15]{GH}
David Gepner and Rune Haugseng.
\newblock {Enriched $\infty$-categories via non-symmetric $\infty$-operads}.
\newblock {\em Adv. Math.}, 279:575--716, 2015.

\bibitem[Hau21]{haugseng2021lax}
Rune Haugseng.
\newblock {On lax transformations, adjunctions, and monads in
  $(\infty,2)$-categories}.
\newblock {\em High. Struct.}, 5(1):244--281, 2021.

\bibitem[Hei20]{Hei-arxiv}
Hadrian Heine.
\newblock {An equivalence between enriched $\infty$-categories and
  $\infty$-categories with weak action}.
\newblock 2020.
\newblock \href{https://arxiv.org/abs/2009.02428v5}{arXiv:2009.02428v5}
  [math.AT].

\bibitem[Hei23]{Hei}
Hadrian Heine.
\newblock {An equivalence between enriched $\infty$-categories and
  $\infty$-categories with weak action}.
\newblock {\em Adv. Math.}, 417:108941, 2023.

\bibitem[Hei24]{BiEnr}
Hadrian Heine.
\newblock {On bi-enriched $\infty$-categories}.
\newblock 2024.
\newblock \href{https://arxiv.org/abs/2406.09832v1}{arXiv:2406.09832v1}
  [math.CT].

\bibitem[HHLN23]{HHLN}
Rune Haugseng, Fabian Hebestreit, Sil Linskens, and Joost Nuiten.
\newblock {Two-variable fibrations, factorisation systems and $\infty
  $-categories of spans}.
\newblock {\em Forum Math. Sigma}, 11:e111, 2023.

\bibitem[Hin20]{Hin}
Vladimir Hinich.
\newblock {Yoneda lemma for enriched $\infty$-categories}.
\newblock {\em Adv. Math.}, 367:107129, 2020.

\bibitem[Lur09]{HTT}
Jacob Lurie.
\newblock {\em {Higher Topos Theory}}.
\newblock Ann. Math. Stud. Princeton University Press, 2009.

\bibitem[Lur17]{HA}
Jacob Lurie.
\newblock {Higher Algebra}.
\newblock \url{https://www.math.ias.edu/~lurie/papers/HA.pdf}, 2017.

\bibitem[Ram23]{Ram}
Maxime Ramzi.
\newblock {An elementary proof of the naturality of the Yoneda embedding}.
\newblock {\em Proc. Am. Math. Soc.}, 151(10):4163--4171, 2023.

\end{thebibliography}

\end{document}